\documentclass{article}

\usepackage{amsmath,amssymb,amsfonts}%
\usepackage{theorem}%
\usepackage{a4wide}%
\usepackage{color}%
\usepackage{graphicx}%
\usepackage{hyperref}%

\theoremstyle{change}%

\newtheorem{definition}{Definition:}[section]%
\newtheorem{theorem}[definition]{Theorem:}%
\newtheorem{proposition}[definition]{Proposition:}%
\newtheorem{lemma}[definition]{Lemma:}%
\newtheorem{corollary}[definition]{Corollary:}%
{\theorembodyfont{\rmfamily} }%
{\theorembodyfont{\rmfamily} }%

\newenvironment{proof}
  {{\bf Proof:}}
  {\qquad \hspace*{\fill} $\Box$}%
  
\newcommand{\UC}{\mathcal{U}}%
\newcommand{\DC}{\mathcal{D}}%
\newcommand{\EC}{\mathcal{E}}%
\newcommand{\FC}{\mathcal{F}}%
\newcommand{\LC}{\mathcal{L}}%
\newcommand{\OC}{\mathcal{O}}%
\newcommand{\PC}{\mathcal{P}}%
\newcommand{\SC}{\mathcal{S}}%
\newcommand{\tm}{\times}%
\newcommand{\R}{\mathbb{R}}%
\newcommand{\rmd}{\mathrm{d}}%
\newcommand{\inv}{\mathrm{inv}}%
\newcommand{\cl}{\mathrm{cl}}%
\newcommand{\ep}{\varepsilon}%
\newcommand{\N}{\mathbb{N}}%
\newcommand{\Z}{\mathbb{Z}}%
\newcommand{\rme}{\mathrm{e}}%
\newcommand{\inner}{\mathrm{int}}%
\newcommand{\Mo}{\mathrm{Mo}}%
\newcommand{\Ly}{\mathrm{Ly}}%
\newcommand{\Fl}{\mathrm{Fl}}%
\newcommand{\dist}{\operatorname{dist}}%

\begin{document}

\title{Robustness of critical bit rates for practical stabilization of networked control systems}

\author{A. Da Silva\thanks{Imecc - Unicamp, Departamento de Matem\'atica, Rua S\'ergio Buarque de Holanda, 651, Cidade Universit\'aria Zeferino Vaz 13083-859, Campinas - SP, Brasil (e-mail: ajsilva@ime.unicamp.br)} \and C. Kawan\thanks{Faculty of Computer Science and Mathematics, University of Passau, 94032 Passau, Germany (e-mail: christoph.kawan@uni-passau.de)}}
\date{}%

\maketitle

\begin{abstract}
In this paper we address the question of robustness of critical bit rates for the stabilization of networked control systems over
digital communication channels. For a deterministic nonlinear system, the smallest bit rate above which practical stabilization (in the sense of set-invariance) can be achieved is measured by the invariance entropy of the system. Under the assumptions of chain controllability and a uniformly hyperbolic structure on the set of interest, we prove that the invariance entropy varies continuously with respect to system parameters. Hence, in this case the critical bit rate is robust with respect to small perturbations.%
\end{abstract}


%%\vspace{0.1cm}%

% 	93C10   	Nonlinear systems
% 	93B35   	Sensitivity (robustness)
%	 	93D09   	Robust stability
%		37D20   	Uniformly hyperbolic systems (expanding, Anosov, Axiom A, etc.)

\section{Introduction}

In networked control systems, the communication between sensors, controllers and actuators is accomplished through a shared digital communication network. There are several aspects of such networks which put severe constraints on the available data rates. In the first place, the digital nature of the communication channels puts a limit on the number of bits that can be transmitted reliably in one unit of time. This naturally leads to the problem of determining the smallest channel capacity or bit rate above which a certain control objective such as stabilization can be achieved. Numerous authors have studied this problem both in deterministic and stochastic setups, for a variety of control objectives, under different assumptions on the network topologies and on the coding and control policies, see, e.g., the papers \cite{Col,CKa,DKa,LHe,LMi,MP1,Nea}, the monographs \cite{BYu,Ka1,MSa} and the survey papers \cite{FMi,Ne2}. In many of these works, expressions or estimates of the critical capacities in terms of dynamical entropies or Lyapunov exponents have been obtained.%

In practice, there are always unknown parameters in the system under consideration. Therefore, one important issue is the robustness of the critical bit rates under variation of system parameters. Since both entropy and Lyapunov exponents as functions of the dynamical system are known to have jump discontinuities, one cannot expect that critical bit rates behave robustly without appropriate assumptions on the system under consideration. To the best of our knowledge, this issue so far has only been addressed in \cite{MP1} for state estimation objectives. In the paper at hand, we identify a setup in which the desired robustness property is satisfied for the problem of practical stabilization (i.e., set-invariance).% 

The paper \cite{Nea} introduced the notion of topological feedback entropy as a measure for the smallest rate of information above which a compact subset of the state space can be rendered invariant by a controller which receives the state information via a noiseless discrete channel. An equivalent notion, called invariance entropy, was introduced in \cite{CKa}. The monograph \cite{Ka1} presents the foundations of a theory which aims at a characterization of invariance entropy in terms of dynamical quantities such as Lyapunov exponents and escape rates. This works particularly well under the assumption that the subset to be stabilized has a uniformly hyperbolic structure. In fact, for a uniformly hyperbolic chain control set of a control-affine system, the paper \cite{DSK} provides a closed expression of the invariance entropy in terms of Lyapunov exponents. Uniformly hyperbolic chain control sets are also known to vary continuously in the Hausdorff metric under variation of system parameters, cf.~\cite{CDu}. We use these results to prove that the invariance entropy of a uniformly hyperbolic chain control set varies continuously with respect to system parameters. Uniformly hyperbolic chain control sets arise around hyperbolic equilibrium points when the control range is sufficiently small and certain regularity assumptions are satisfied (as, e.g., controllability of the linearization at the equilibrium). In \cite{CLe} an example of a stirred tank reactor is studied, where this happens. A large class of algebraic examples for uniformly hyperbolic chain control sets was identified in \cite{DS2}. %

The paper is organized as follows. In Section \ref{sec_prelim}, we give a review of control-affine systems, additive cocycles, concepts of semicontinuity and the shadowing property of uniformly hyperbolic systems. In Section \ref{sec_datarate}, we provide a new justification that the invariance entropy is a measure for the smallest data rate above which a set can be rendered invariant by a symbolic coding and control scheme. Section \ref{sec_spectra} studies semicontinuity properties of spectral sets for additive and subadditive cocycles over control flows of parametrized control-affine systems. Finally, we prove our main result about the continuity of the invariance entropy on a uniformly hyperbolic chain control set in Section \ref{sec_mainres}. Here we essentially use the following three facts: the property of uniform hyperbolicity is robust under $C^1$-perturbations, the invariance entropy of a uniformly hyperbolic chain control set can be expressed as the infimum of the Morse spectrum of an additive cocycle and, by the results of Section \ref{sec_spectra}, this infimum depends continuously on parameters.%

%In this paper we address some issues about minimal channel capacities or data rates necessary for stabilization, which are relevant from the applied point of view. Many works have shown that the minimal %channel capacities for stabilization and for observability of nonlinear systems (both deterministic and stochastic) are intimately related to dynamical entropy notions. In case of observability, under %appropriate assumptions, the minimal capacity is in fact equal to the topological entropy of the system. From the applied point of view now several questions arise:%
%\begin{enumerate}
%\item[(1)] How can the minimal capacity or the corresponding entropy be actually computed?%
%\item[(2)] What happens to the minimal capacity under small perturbations of the system?%
%\item[(3)] How can one construct a system which operates close to the theoretical capacity limit?%
%\end{enumerate}
%Our aim is to address questions (1) and (2) for the problem of set stabilization in deterministic systems. In this case, the channel is modelled noisefree and can be considered as a bit-pipe, allowing %for the transmission of a certain number $R$ of bits per time unit.%

\section{Preliminaries}\label{sec_prelim}

{\bf Notation.} All manifolds considered in this paper are smooth, i.e., equipped with a $C^{\infty}$ differentiable structure. If $f:M \rightarrow N$ is a differentiable map between smooth manifolds $M$ and $N$, then $\rmd f(x):T_xM \rightarrow T_{f(x)}N$ denotes its derivative at $x\in M$. On a Riemannian manifold $M$, we always write $d(x,y)$ for the geodesic distance of two points $x,y \in M$. The norm on each tangent space $T_xM$ is simply denoted by $|\cdot|$. We write $\inner A$ and $\cl A$ for the interior and closure of a set $A$, respectively. By $\dist(x,A)$ we denote $\inf_{y\in A}d(x,y)$. If $u_1:[0,\tau_1] \rightarrow U$ and $u_2:[0,\tau_2]\rightarrow U$ are function, we write $u_1u_2:[0,\tau_1+\tau_2]\rightarrow U$ for their concatenation, i.e., $(u_1u_2)(t) = u_1(t)$ for $t\in[0,\tau_1]$ and $(u_1u_2)(t) = u_2(t-\tau_1)$ for $t\in[0,\tau_2]$. We also write $u^n$ for the concatenation of $n$ copies of $u$.%

\subsection{Control-affine systems}\label{subsec_cas}%

A \emph{control-affine system} is a control system given by differential equations of the form%
\begin{equation}\label{eq_cas}
  \dot{x} = f_0(x) + \sum_{i=1}^m u_i(t)f_i(x),\quad u\in\UC,%
\end{equation}
where $f_0,f_1,\ldots,f_m$ are smooth vector fields on a Riemannian manifold $M$ and $\UC = L^{\infty}(\R,U)$ for a compact and convex set $U \subset \R^m$ with $0\in\inner U$.%

By $\varphi(t,x,u)$ we denote the unique solution of \eqref{eq_cas} for the control function $u\in\UC$, satisfying the initial condition $\varphi(0,x,u) = x$. For simplicity (in fact without loss of generality), we assume that all solutions are defined on $\R$, and hence we obtain a \emph{transition map}%
\begin{equation*}
  \varphi:\R \tm M \tm \UC \rightarrow M,\quad (t,x,u) \mapsto \varphi(t,x,u).%
\end{equation*}
We also write $\varphi_{t,u}(x) = \varphi(t,x,u)$. The set $\UC$ of admissible control functions becomes a compact metrizable space with the weak$^*$-topology of $L^{\infty}(\R,\R^m) = L^1(\R,\R^m)^*$ and $\varphi$ can be extended to a continuous skew-product flow%
\begin{equation*}
  \Phi:\R \tm (\UC \tm M) \rightarrow \UC \tm M,\quad \Phi_t(u,x) = (\theta_tu,\varphi(t,x,u)),%
\end{equation*}
called the \emph{control flow} of the control system \eqref{eq_cas}. Here $\theta:\R\tm\UC\rightarrow\UC$, $\theta_tu = u(t+\cdot)$, is the \emph{shift flow} on $\UC$. These general facts can be found in \cite{CKl}.%

The \emph{set of points reachable from $x$ up to time $T$} is defined as%
\begin{equation*}
  \OC^+_{\leq T}(x) := \left\{ y \in M\ :\ \exists u\in\UC,\ t \in [0,T] \mbox{ with } y = \varphi(t,x,u) \right\},%
\end{equation*}
and $\OC^+(x) := \bigcup_{T>0}\OC_{\leq T}^+(x)$ is the \emph{positive orbit of $x$}. With $\OC^-_{\leq T}(x)$ and $\OC^-(x)$ we denote the corresponding sets for the time-reversed system. We say that system \eqref{eq_cas} is \emph{locally accessible from $x$} if $\inner \OC^{\pm}_{\leq T}(x) \neq \emptyset$ for all $T>0$. A sufficient condition for local accessibility is the \emph{Lie algebra rank condition}. This condition is satisfied at $x\in M$ if the Lie algebra $\LC$ generated by the vector fields $f_0,f_1,\ldots,f_m$ satisfies $\LC(x) = \{ f(x) : f \in \LC\} = T_xM$.%

A set $D\subset M$ is a \emph{control set} of \eqref{eq_cas} if it is maximal w.r.t.~set inclusion with the following properties:%
\begin{enumerate}
\item[(i)] $D$ is \emph{controlled invariant}, i.e., for each $x\in D$ there is $u\in\UC$ with $\varphi(\R_+,x,u) \subset D$.%
\item[(ii)] \emph{Approximate controllability} holds on $D$, i.e., $D \subset \cl\OC^+(x)$ for all $x\in D$.%
\end{enumerate}
The \emph{lift of a control set} $D$ to $\UC \tm M$ is defined by%
\begin{equation*}
  \DC := \cl\left\{ (u,x) \in \UC \tm M\ :\ \varphi(\R,x,u) \subset \inner D\right\}.%
\end{equation*}
Control sets with nonempty interior have the \emph{no-return property}: If $x \in D$ and $\varphi(\tau,x,u) \in D$ for some $\tau>0$ and $u\in\UC$, then $\varphi(t,x,u) \in D$ for all $t\in[0,\tau]$.%

A \emph{chain control set} of \eqref{eq_cas} is a set $E \subset M$ which is maximal with the following properties:%
\begin{enumerate}
\item[(i)] $E$ is \emph{full-time controlled invariant}, i.e., for each $x\in E$ there is $u\in\UC$ with $\varphi(\R,x,u) \subset E$.%
\item[(ii)] \emph{Chain controllability} holds on $E$, i.e., for each two $x,y\in E$ and all $\ep,T>0$ there are $n\in\N$, controls $u_0,\ldots,u_{n-1}\in\UC$, states $x_0 = x,x_1,\ldots,x_{n-1},x_n = y$ and times $t_0,\ldots,t_{n-1}\geq T$ such that%
\begin{equation*}
  d(\varphi(t_i,x_i,u_i),x_{i+1}) < \ep,\quad i = 0,1,\ldots,n-1.%
\end{equation*}
\end{enumerate}
The \emph{lift of a chain control set} $E$ to $\UC \tm M$ is defined by%
\begin{equation*}
  \EC := \left\{ (u,x) \in \UC \tm M\ :\ \varphi(\R,x,u) \subset E \right\}.%
\end{equation*}
The set $\EC$ is a closed invariant set of the control flow and it is compact if $E$ is compact. If $D$ is a control set with nonempty interior and local accessibility holds on $\inner D$, then $D$ is contained in a unique chain control set (see \cite[Ch.~4]{CKl}). Moreover, the lifts of the chain control sets are precisely the maximal invariant chain transitive sets of the control flow (see Subsect.~\ref{subsec_cocycles} for the definition of chain transitivity).%

A compact chain control set $E$ is called \emph{uniformly hyperbolic} if there exists a decomposition%
\begin{equation*}
  T_xM = E^-_{u,x} \oplus E^+_{u,x}\quad \forall (u,x)\in\EC%
\end{equation*}
with subspaces $E^{\pm}_{u,x}$ satisfying%
\begin{enumerate}
\item[(H1)] $\rmd\varphi_{t,u}(x)E^{\pm}_{u,x} = E^{\pm}_{\Phi_t(u,x)}$ for all $(u,x)\in\EC$ and $t\in\R$.%
\item[(H2)] There exist constants $0 < c \leq 1$ and $\lambda>0$ such that for all $(u,x)\in\EC$,%
\begin{eqnarray*}
  \left|\rmd\varphi_{t,u}(x)v\right| &\leq& c^{-1}\rme^{-\lambda t}|v| \mbox{\quad for all\ } t\geq0,\ v\in E^-_{u,x},\\
  \left|\rmd\varphi_{t,u}(x)v\right| &\geq& c\rme^{\lambda t}|v| \mbox{\quad for all\ } t\geq0,\ v\in E^+_{u,x}.%
\end{eqnarray*}
\end{enumerate}
From (H1) and (H2) it follows that $E^{\pm}_{u,x}$ depend continuously on $(u,x)$, cf.~\cite[Ch.~6]{Ka1}. We write%
\begin{equation*}
  J^+\rmd\varphi_{t,u}(x) := \bigl|\det\rmd\varphi_{t,u}(x)|_{E^+_{u,x}}:E^+_{u,x} \rightarrow E^+_{\Phi_t(u,x)}\bigr|%
\end{equation*}
for the unstable determinant of $\rmd\varphi_{t,u}$ and note that%
\begin{equation}\label{eq_unstabledet_cocycle}
  \gamma_t(u,x) := \log J^+\rmd\varphi_{t,u}(x),\quad \gamma:\R \tm (\UC \tm M) \rightarrow \R,%
\end{equation}
defines a continuous additive cocycle over $\EC$ for the control flow $\Phi$, i.e., it satisfies%
\begin{equation*}
  \gamma_{t+s}(u,x) = \gamma_t(u,x) + \gamma_s(\Phi_t(u,x)) \mbox{\quad for all\ } t,s\in\R,\ (u,x) \in \EC.%
\end{equation*}

Let $K,Q \subset M$ be sets so that $K$ is compact and for each $x\in K$ there is $u\in\UC$ with $\varphi(\R_+,x,u) \subset Q$. Then we call $(K,Q)$ an \emph{admissible pair} and define its \emph{invariance entropy} $h_{\inv}(K,Q)$ as follows. A set $\SC\subset\UC$ is called $(\tau,K,Q)$-spanning (for some $\tau>0$) if for each $x\in K$ there is $u\in\SC$ with $\varphi([0,\tau],x,u) \subset Q$. The minimal cardinality of such a set is denoted by $r_{\inv}(\tau,K,Q)$, where the value $+\infty$ is allowed. Then%
\begin{equation*}
  h_{\inv}(K,Q) := \limsup_{\tau\rightarrow\infty}\frac{1}{\tau}\log r_{\inv}(\tau,K,Q),%
\end{equation*}
where $\log$ stands for the natural logarithm. We have the following result, cf.~\cite{DSK}.%

\begin{theorem}\label{thm_ie_formula}
Consider a uniformly hyperbolic chain control set $E$ of system \eqref{eq_cas} with nonempty interior. Additionally assume that%
\begin{enumerate}
\item[(i)] the Lie algebra rank condition holds on $\inner E$ and%
\item[(ii)] for each $u\in\UC$ there exists a unique $x\in E$ with $(u,x) \in \EC$, i.e., $\EC$ is a graph over $\UC$.%
\end{enumerate}
Then $E$ is the closure of a control set $D$ and for every compact set $K \subset D$ with positive volume,%
\begin{equation*}
  h_{\inv}(K,E) = \inf_{(u,x) \in \EC}\limsup_{t\rightarrow\infty}\frac{1}{t}\log J^+\rmd\varphi_{t,u}(x).%
\end{equation*}
\end{theorem}

To make it easier to refer to the assumptions in this theorem, we introduce the following notion.%

\begin{definition}
Let $E$ be a compact chain control set of a control-affine system. We say that $E$ is {\bf regularly uniformly hyperbolic} if the following hold:%
\begin{enumerate}
\item[(a)] $E$ is uniformly hyperbolic.%
\item[(b)] $\inner E \neq \emptyset$ and the Lie algebra rank condition holds on $\inner E$.%
\item[(c)] $\EC$ is a graph over $\UC$.%
\item[(d)] $\EC$ is an isolated invariant set of the control flow.%
\end{enumerate}
\end{definition}

The meaning of condition (d) will become clear later. The following theorem provides a sufficient condition for assumption (ii) in Theorem \ref{thm_ie_formula} to be satisfied, cf.~\cite{Ka2}.%

\begin{theorem}\label{thm_str}
Let $E$ be a uniformly hyperbolic chain control set of system \eqref{eq_cas} and assume that $\EC$ is an isolated invariant set of $\Phi$. Let $u_0$ be a constant control function with value in $\inner U$ and suppose that the following hypotheses are satisfied:%
\begin{enumerate}
\item[(i)] The Lie algebra rank condition holds on $E$.%
\item[(ii)] For each $x$ with $(u_0,x)\in\EC$ and each $\rho\in(0,1]$ it holds that $x\in\inner\OC^+_{\rho}(x)$, where $\OC^+_{\rho}(x)=\{\varphi(t,x,u) : t\geq 0,\ u\in\UC^{\rho}\}$ with
\begin{equation*}
  \UC^{\rho} = \{u \in \UC\ :\ u(t) \in u_0 + \rho(U - u_0) \mbox{ a.e.}\}.%
\end{equation*}
\end{enumerate}
Then $\EC$ is the graph of a continuous function $\UC \rightarrow E$.%
\end{theorem}

\subsection{Cocycles and spectra}\label{subsec_cocycles}%

Let $(X,d)$ be a metric space and $\Phi:\R \tm X \rightarrow X$, $(t,x)\mapsto\Phi_t(x)$, a continuous flow. A set $A \subset X$ is $\Phi$-invariant if $\Phi_t(x) \in A$ whenever $x\in A$ and $t\in\R$. If $A \subset X$ is a compact $\Phi$-invariant set, a continuous map $\gamma:\R \tm A \rightarrow \R$, $(t,x) \mapsto \gamma_t(x)$, is called an \emph{additive cocycle for $\Phi$ over $A$} if%
\begin{equation*}
  \gamma_{t+s}(x) = \gamma_t(x) + \gamma_s(\Phi_t(x)) \mbox{\quad for all\ } t,s\in\R,\ x\in A.%
\end{equation*}
Any limit of the form%
\begin{equation}\label{eq_deflyapexp}
  \lambda(x;\gamma) := \lim_{t\rightarrow\infty}\frac{1}{t}\gamma_t(x),\quad x\in A,%
\end{equation}
is called a \emph{$\gamma$-Lyapunov exponent}. The \emph{Lyapunov spectrum of $\gamma$ over $A$} is defined by%
\begin{equation*}
  \Sigma_{\Ly}(\gamma,A) := \left\{\lambda \in \R\ :\ \lambda = \lambda(x;\gamma) \mbox{ for some } x \in A\right\}.%
\end{equation*}
We note that the limit in \eqref{eq_deflyapexp} not necessarily exists for every $x\in A$. However, it exists for any periodic point $x$ and in this case satisfies%
\begin{equation*}
  \lambda(x;\gamma) = \frac{1}{T}\gamma_T(x),%
\end{equation*}
where $T>0$ is the period of $x$. For $\ep,T>0$, an \emph{$(\ep,T)$-chain} $\zeta$ for $\Phi$ is given by $n\in\N$, points $x_0,\ldots,x_n \in X$ and times $T_0,\ldots,T_{n-1} \geq T$ so that%
\begin{equation*}
  d(\Phi_{T_i}(x_i),x_{i+1}) < \ep \mbox{\quad for\ } i = 0,1,\ldots,n-1.%
\end{equation*}
If $x_0 = x_n$, the chain $\zeta$ is called \emph{periodic}. The \emph{$\gamma$-Morse exponent} of $\zeta$ is given by%
\begin{equation*}
  \gamma(\zeta) := \frac{1}{\sum_{i=0}^{n-1}T_i}\sum_{i=0}^{n-1}\gamma_{T_i}(x_i).%
\end{equation*}
We say that the chain $\zeta$ belongs to $A$ if the initial point $x_0$ and the endpoint $x_n$ are in $A$. The \emph{Morse spectrum} of $\gamma$ over $A$ is defined by%
\begin{equation*}
  \Sigma_{\Mo}(\gamma,A) := \bigcap_{\ep,T>0} \cl \left\{\gamma(\zeta) : \zeta \mbox{ is an } (\ep,T)\mbox{-chain belonging to } A \right\}.%
\end{equation*}
The set $A$ is called chain transitive if any two points in $A$ can be joined by an $(\ep,T)$-chain for any $\ep,T>0$. We will use the following properties of the Lyapunov and the Morse spectrum, which hold whenever the restriction of $\Phi$ to $A$ is chain transitive and $A$ is compact (cf.~\cite{KSt,SMS}):%
\begin{enumerate}
\item[(1)] The Morse spectrum $\Sigma_{\Mo}(\gamma,A)$ is a compact interval.%
\item[(2)] $\Sigma_{\Mo}(\gamma,A)$ is the smallest closed interval containing $\Sigma_{\Ly}(\gamma,A)$.%
\item[(3)] The left and right endpoints of $\Sigma_{\Mo}(\gamma,A)$ are $\gamma$-Lyapunov exponents.%
\item[(4)] Periodic $(\ep,T)$-chains are sufficient to obtain $\Sigma_{\Mo}(\gamma,A)$, i.e.,%
\begin{equation*}
  \Sigma_{\Mo}(\gamma,A) := \bigcap_{\ep,T>0} \cl \left\{\gamma(\zeta) : \zeta \mbox{ is a periodic } (\ep,T)\mbox{-chain belonging to } A \right\}.%
\end{equation*}
\end{enumerate}
We also recall the notion of a \emph{subadditive cocycle for $\Phi$ over $A$}. This is a continuous map $\kappa:\R\tm A \rightarrow \R$ so that%
\begin{equation*}
  \kappa_{t+s}(x) \leq \kappa_t(x) + \kappa_s(\Phi_t(x)) \mbox{\quad for all\ } t,s\in\R,\ x\in A.%
\end{equation*}

\subsection{Upper and lower semicontinuity}

A map $f:X \rightarrow \R$, defined on a metric space $(X,d)$, is called \emph{upper semicontinuous} at $x_0 \in X$ if%
\begin{equation*}
  \limsup_{x\rightarrow x_0}f(x) \leq f(x_0).%
\end{equation*}
It is called \emph{lower semicontinuous} at $x_0$ if%
\begin{equation*}
  \liminf_{x\rightarrow x_0}f(x) \geq f(x_0).%
\end{equation*}
Obviously, $f$ is continuous at $x_0$ iff it is both upper and lower semicontinuous at $x_0$.%

The concepts of upper and lower semicontinuity are also defined for set-valued mappings (see, e.g., \cite{AFr}). Let $(X,d_X)$ and $(Y,d_Y)$ be metric spaces and $f:X \rightarrow \PC_0(Y)$ a map from $X$ into the set of nonempty subsets of $Y$. Then $f$ is called \emph{upper semicontinuous} at $x_0\in X$ if for every $\ep>0$ there is $\delta>0$ so that $d_X(x,x_0)<\delta$ implies $\sup_{y\in f(x)}\dist_Y(y,f(x_0))<\ep$. It is called \emph{lower semicontinuous} at $x_0$ if for every $\ep>0$ there is $\delta>0$ so that $d_X(x,x_0)<\delta$ implies $\sup_{y\in f(x_0)}\dist_Y(y,f(x)) < \ep$.%

The set-valued map $f$ is continuous w.r.t.~the Hausdorff metric on $\PC_0(Y)$ at $x_0$ iff it is lower and upper semicontinuous at $x_0$. Moreover, if $f$ has closed values and $Y$ is compact, then upper semicontinuity is equivalent to%
\begin{equation*}
  f(x_0) \supset \limsup_{x\rightarrow x_0}f(x) := \left\{y\in Y\ :\ \exists x_k \rightarrow x \mbox{ in } X \mbox{ and } y_k \in f(x_k) \mbox{ with } y_k \rightarrow y\right\}.%
\end{equation*}
Moreover, if $f$ is compact-valued, lower semicontinuity is equivalent to%
\begin{equation*}
  f(x_0) \subset \liminf_{x\rightarrow x_0}f(x) := \left\{y\in Y\ :\ \forall x_k \rightarrow x \mbox{ in } X\ \exists y_k \in f(x_k) \mbox{ with } y_k \rightarrow y\right\}.%
\end{equation*}

\subsection{The shadowing lemma for skew-product systems}\label{subsec_shadowing}

In the following, we recall the shadowing lemma \cite[Lem.~2.11]{MZh} for discrete-time skew-product systems. Let $B$ be a compact metric space and $M$ a Riemannian manifold. Let $f:B\tm M \rightarrow B\tm M$ be a homeomorphism of the form $f(b,x) = (V(b),\Psi(b,x))$, where $V:B\rightarrow B$ is also a homeomorphism. Suppose that the map $\Psi_b := \Psi(b,\cdot):M\rightarrow M$ is a diffeomorphism for each $b\in B$ whose derivative depends continuously on $(b,x)$. We write $O(b,x) = \{f^k(b,x) : k\in\Z\}$ for the orbit through $(b,x)$. A sequence $(b_n,x_n)_{n\in\Z}$ in $B\tm M$ is called a \emph{$\delta$-pseudo-orbit} if%
\begin{equation*}
  b_{n+1} = V(b_n) \mbox{\quad and\quad} d(f(b_n,x_n),(b_{n+1},x_{n+1})) < \delta \mbox{\quad for all\ } n\in\Z,%
\end{equation*}
where $d(\cdot,\cdot)$ is any fixed metric on $B\tm M$. The pseudo-orbit $(b_n,x_n)_{n\in\Z}$ is \emph{$\ep$-shadowed} by an orbit $O(b,x)$ if%
\begin{equation*}
  b = b_0 \mbox{\quad and\quad} d(f^n(b,x),(b_n,x_n)) < \ep \mbox{\quad for all\ } n\in\Z.%
\end{equation*}
A closed $f$-invariant set $\Lambda \subset B\tm M$ is called \emph{uniformly hyperbolic} if there are constants $C \geq 1$ and $\lambda \in (0,1)$ and a splitting%
\begin{equation*}
  T_xM = E^+_{b,x} \oplus E^-_{b,x} \mbox{\quad for all\ } (b,x) \in \Lambda%
\end{equation*}
such that for all $(b,x)\in\Lambda$ the following conditions holds:%
\begin{enumerate}
\item[(1)] $\rmd\Psi_b(x) E^{\pm}_{b,x} = E^{\pm}_{f(b,x)}$.%
\item[(2)] For all $n\geq0$,%
\begin{eqnarray*}
  |\rmd\Psi_b(x) f^n(b,x) v| &\leq& C\lambda^n |v| \mbox{\quad for all } v \in E^-_{b,x},\\
  |\rmd\Psi_b(x) f^n(b,x) v| &\geq& C^{-1}\lambda^{-n} |v| \mbox{\quad for all } v \in E^+_{b,x}.%
\end{eqnarray*}
\end{enumerate}
The set $\Lambda$ is called \emph{isolated invariant} if there is a neighborhood $W$ of $\Lambda$ in $B\tm M$ so that $O(b,x) \subset W$ implies $O(b,x) \subset \Lambda$. Equivalently, $\Lambda$ is the largest invariant set in $W$.%

The shadowing lemma reads as follows.%

\begin{lemma}\label{lem_shadowing}
Let $\Lambda$ be a uniformly hyperbolic set of $f:B\tm M \rightarrow B\tm M$. Then there are neighborhoods $U$ of $\Lambda$ and $W$ of $f$ (in the $C^{0,1}$-topology) satisfying the following properties:%
\begin{enumerate}
\item[(i)] For any $g\in W$ and any $\ep>0$ there is $\delta>0$ so that every $\delta$-pseudo-orbit $(b_n,x_n)_{n\in\Z}$ of $g$ in $U$ is $\ep$-shadowed by an orbit $O(b,x)$ of $g$.%  
\item[(ii)] There is an $\ep_0>0$ so that for $0 < \ep < \ep_0$ the $\ep$-shadowing orbit is uniquely determined by the $\delta$-pseudo-orbit.% 
\item[(iii)] If $\Lambda$ is an isolated invariant set for $f$, then any $g\in W$ has an isolated invariant uniformly hyperbolic set $\Delta \subset U$ and the shadowing orbit $O(b,x)$ is in $\Delta$.%
\end{enumerate}
\end{lemma}

Note that in \cite{MZh} it is additionally assumed that the homeomorphism $V:B \rightarrow B$ is almost periodic. However, this is not used in the proof of the shadowing lemma.%

\section{A data-rate theorem}\label{sec_datarate}

In this section, we give a new justification that the invariance entropy $h_{\inv}(K,Q)$ is a measure for the smallest bit rate above which the set $Q$ can be stabilized, assuming that $Q$ is the closure of a control set $D$ and $K \subset D$. In the case $K = Q$ this holds for arbitrary controlled invariant sets $Q$ (see \cite{Ka1}). If $K$ is a proper subset of $Q$, however, in general only the inequality $h_{\inv}(K,Q) \leq h_{\inv}(Q,Q)$ holds, implying that $h_{\inv}(K,Q)$ is a lower bound on the critical bit rate.%

The idea employed here is to consider the stabilization objective to keep the system in $K$ on the long run without leaving $Q$, whenever the initial state is in $K$ (instead of keeping it in $Q$ for arbitrary initial values $x_0 \in Q$).%

\begin{definition}
Let $D$ be a control set of system \eqref{eq_cas} with closure $Q = \cl D$ and let $K \subset D$ be a compact set with nonempty interior. For $\tau>0$, a set $\SC \subset \UC$ is called $(\tau,K)^Q$-spanning if for each $x\in K$ there exists $u\in\SC$ with $\varphi(\tau,x,u) \in K$. We write $r_{\inv}(\tau,K)^Q$ for the minimal cardinality of such a set and define%
\begin{eqnarray*}
  h_{\inv}(K)^Q := \lim_{\tau\rightarrow\infty}\frac{1}{\tau}\log r_{\inv}(\tau,K)^Q,\quad h_{\inv}^*(Q) := \sup_{K \subset D}h_{\inv}(K)^Q.%
\end{eqnarray*}
\end{definition}

We observe that by the no-return property of control sets each trajectory $\varphi(t,x,u)$ in the above definition satisfies $\varphi([0,\tau],x,u) \subset D$. The following proposition shows that the definition is meaningful.%

\begin{proposition}\label{prop_iedefcorrect}
Let $D$ be a control set with closure $Q = \cl D$. Assume that the system is locally accessible on $\inner D$. Then, for any sufficiently large compact set $K \subset D$ with nonempty interior there exists $\tau_0 > 0$ so that for all $\tau \geq \tau_0$ a finite $(\tau,K)^Q$-spanning set exists. Moreover, the function $\tau \mapsto \log r_{\inv}(\tau,K)^Q$ is subadditive, and hence the limit in the definition of $h_{\inv}(K)^Q$ exists.%
\end{proposition}

\begin{proof}
By local accessibility, for any $x_0 \in \inner D$ a controlled periodic trajectory $(\varphi(\cdot,x_0,u_0),u_0(\cdot))$ exists such that $\varphi(t,x_0,u_0) \in \inner D$ for all $t\in\R$. Fix such a trajectory for one $x_0 \in \inner D$ and let $K \subset D$ be a compact set with $\varphi(\R,x_0,u_0) \subset \inner K$. From the assumption of local accessibility complete controllability on $\inner D$ follows. Hence, for every $x\in K$ there exists a control function $u_x$ and a time $t_x > 0$ with $\varphi(t_x,x,u_x) = x_0$. Moreover, we can assume that $t_x$ is uniformly bounded by \cite[Lem.~3.2.21]{CKl}. We concatenate $u_x:[0,t_x] \rightarrow U$ with $u_0:\R \rightarrow U$, and denote this new control function again by $u_x$. In this way, we can achieve that $\varphi(\tau,x,u_x) \in K$ for a time $\tau>0$, independent of $x$. Moreover, the time $\tau$ can be chosen arbitrarily large. Since $K$ is compact, every $x\in K$ has a neighborhood $N_x$ so that $\varphi(\tau,y,u_x) \in \inner K$ for all $y\in N_x$. If $m$ is the cardinality of a finite subcover of $\{N_x : x \in K\}$, the corresponding $m$ control functions obviously form a $(\tau,K)^Q$-spanning set. 

To see that the limit in the definition of $h_{\inv}(K)^Q$ exists, observe that $\tau \mapsto r_{\inv}(\tau,K)^Q$ is subadditive. If $\SC_1$ is $(\tau_1,K)^Q$-spanning and $\SC_2$ is $(\tau_2,K)^Q$-spanning, then the set consisting of all possible concatenations of the controls in $\SC_1$ and $\SC_2$ is $(\tau_1+\tau_2,K)^Q$-spanning, implying $r_{\inv}(\tau_1+\tau_2,K)^Q \leq r_{\inv}(\tau_1,K)^Q \cdot r_{\inv}(\tau_2,K)^Q$.%
\end{proof}

The following theorem shows that the formula for $h_{\inv}(K,Q)$ in Theorem \ref{thm_ie_formula} also holds for $h_{\inv}^*(Q)$.%

\begin{theorem}\label{thm_iedr_1}
If $Q$ is a regularly uniformly hyperbolic chain control set, then%
\begin{equation}\label{eq_iestar_formula}
  h_{\inv}^*(Q) = \inf_{(u,x) \in \EC}\limsup_{t\rightarrow\infty}\frac{1}{t}\log J^+\rmd\varphi_{t,u}(x).%
\end{equation}
In fact, the same formula holds for $h_{\inv}(K)^Q$, where $K \subset D$ is an arbitrary compact set with nonempty interior.%
\end{theorem}

\begin{proof}
First observe that $h_{\inv}^*(Q)$ is well-defined, because there exists a control set $D$ with $\cl D = Q$ by Theorem \ref{thm_ie_formula}. Because of the no-return property of control sets, it is clear that a $(\tau,K)^Q$-spanning set is $(\tau,K,Q)$-spanning, hence Theorem \ref{thm_ie_formula} implies the inequality%
\begin{equation*}
  h_{\inv}(K)^Q \geq \inf_{(u,x) \in \EC}\limsup_{t\rightarrow\infty}\frac{1}{t}\log J^+\rmd\varphi_{t,u}(x),%
\end{equation*}
for all $K \subset D$ with nonempty interior (and hence positive volume). Consequently, the same inequality holds for $h_{\inv}^*(Q)$. To show the converse inequality, we have to recall some arguments used in the proof of Theorem \ref{thm_ie_formula}. The upper estimate on $h_{\inv}(K,Q)$ is based on the inequality%
\begin{equation}\label{eq_ie_perest}
  h_{\inv}(K,Q) \leq \lim_{t\rightarrow\infty}\frac{1}{t}\log J^+\rmd\varphi_{t,u}(x) = \frac{1}{\tau}\log J^+\rmd\varphi_{\tau,u}(x)%
\end{equation}
for $\tau$-periodic trajectories corresponding to $(u,x) \in \inner \UC \tm \inner Q$, see \cite[Prop.~5]{DSK}.%

We claim that the same estimate holds for $h_{\inv}(K)^Q$. To see this, we recall that the $(\tau,K,Q)$-spanning sets constructed to prove \eqref{eq_ie_perest}, for sufficiently large $\tau$, lead to trajectories that end in a neighborhood $N$ of the periodic trajectory $\varphi(\cdot,x,u)$. This neighborhood can in fact be chosen arbitrarily small and compact. By \cite[Lem.~3.2.21]{CKl} there exists $T>0$ so that the minimal time to steer from any $x\in N$ to any $y\in K$ is bounded by $T$. Concatenating the control functions in a $(\tau,K,Q)$-spanning set $\SC_0(\tau)$ producing trajectories ending in $N$ with control functions steering those endpoints back to $K$ leads to $(\tau+T,K)^Q$-spanning sets of some cardinality $m \cdot |\SC_0(\tau)|$, where $m$ is a fixed number, independent of $\tau$. (Here we use the same idea as in the proof of Proposition \ref{prop_iedefcorrect} to find a uniform $T$.) This yields%
\begin{equation*}
  h_{\inv}(K)^Q \leq \lim_{\tau\rightarrow\infty} \frac{1}{\tau + T} \log ( m \cdot |\SC_0(\tau)| ) = \lim_{\tau\rightarrow\infty} \frac{1}{\tau} \log |\SC_0(\tau)|,%
\end{equation*}
where the sets $\SC_0(\tau)$ can be chosen so that the latter limit is arbitrarily close to $(1/\tau)\log J^+\rmd\varphi_{\tau,u}(x)$. This proves the claim and shows that%
\begin{equation*}
  h_{\inv}(K)^Q \leq \inf_{(u,x)} \lim_{t\rightarrow\infty}\frac{1}{t}\log J^+\rmd\varphi_{t,u}(x),%
\end{equation*}
the infimum taken over all periodic $(u,x) \in \inner \UC \tm \inner D$. The proof of Theorem \ref{thm_ie_formula} in fact shows that the right-hand side of this inequality is equal to the right-hand side of \eqref{eq_iestar_formula}, completing the proof.%
\end{proof}

Suppose that a sensor measures the states of the system at discrete sampling times $\tau_k = k \tau$ for some $\tau>0$. A coder receiving these measurements generates at each sampling time $\tau_k$ a symbol $s_k$ from a finite coding alphabet $S_k$ of time-varying size. This symbol is transmitted through a digital noiseless channel to a controller. The controller, upon receiving $s_k$, generates an open-loop control $u_k$ on $[0,\tau]$ used as the control input in the time interval $[\tau_k,\tau_{k+1}]$. The aim of this coding and control device is to render a set $K \subset Q$ invariant in the sense of the following definition (here $x_t$ stands for the state at time $t$, resulting from the initial state $x_0$ and the controller actions in the time interval $[0,t]$).%

\begin{figure}[hp]
	\begin{center}
		\includegraphics[width=8.20cm,height=2.50cm,angle=0]{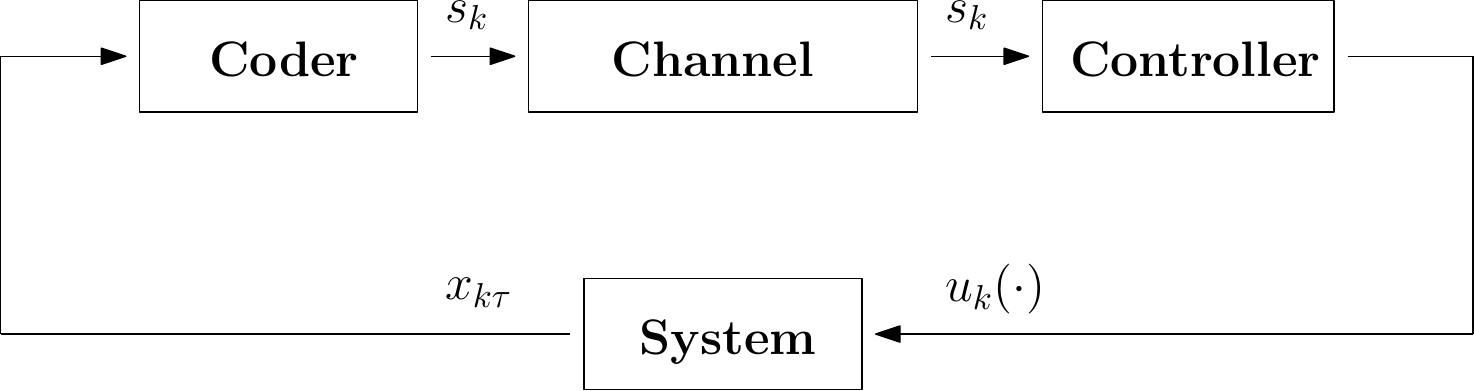}
	\end{center}
\caption{Control over a digital channel}%
\end{figure}

The asymptotic average bit rate of the channel is defined by%
\begin{equation*}
  R := \liminf_{k\rightarrow\infty}\frac{1}{k\tau}\sum_{j=0}^{k-1}\log_2|S_j|.%
\end{equation*}

\begin{definition}
Let $Q \subset M$ be some nonempty set and $K \subset Q$ compact. We say that $K$ is rendered $Q$-invariant by a coder-controller pair if $x_0 \in K$ implies%
\begin{enumerate}
\item[(i)] $x_t \in Q$ for all $t \geq 0$ and%
\item[(ii)] $x_{k\tau} \in K$ for all $k\in\Z_+$, where $\tau$ is the sampling time.%
\end{enumerate}
\end{definition}

\begin{theorem}
Let $D$ be a control set with nonempty interior and put $Q = \cl D$. Assume that local accessibility holds on $\inner D$. Then the smallest bit rate above which a coder-controller pair can be designed that renders a (sufficiently large) compact subset $K \subset D$ $Q$-invariant is given by $\log_2(\rme)h_{\inv}^*(Q)$.%
\end{theorem}

\begin{proof}
Let $\SC$ be a minimal $(\tau,K)^Q$-spanning set for a sufficiently large $\tau>0$. Then, by Proposition \ref{prop_iedefcorrect}, $\SC$ is finite, and hence we can write $\SC = \{u_1,\ldots,u_m\}$ for some $m\in\N$. We define coding regions%
\begin{equation*}
  K_i := \left\{ x\in K\ :\ \varphi(\tau,x,u_i) \in K \right\},\quad i = 1,\ldots,m.%
\end{equation*}
We may assume that the coding regions are disjoint by cutting away overlaps. Now consider the sampling times $\tau_k := k\tau$, $k\in\Z_+$. At time $\tau_k$ let the coder send the symbol $i$ through the channel if and only if $x_{\tau_k} \in K_i$. Let the controller generate the control function $u_i:[0,\tau] \rightarrow U$ upon receiving the symbol $i$. This coding and control scheme certainly guarantees invariance in the desired sense and the asymptotic bit rate is%
\begin{equation*}
  R = \frac{1}{\tau} \log m = \frac{1}{\tau}\log r_{\inv}(\tau,K)^Q.%
\end{equation*}
Since $\tau$ can be chosen arbitrarily, this implies that the infimal bit rate $R_0$ satisfies%
\begin{equation*}
  R_0 \leq h_{\inv}(K)^Q.%
\end{equation*}
Conversely, assume that the control objective is achieved via a channel of bit rate $R$ with a sampling time $\tau>0$. The set $\SC_k$ of all controls generated by the controller in the time interval $[0,k\tau]$, when $x_0$ ranges through $K$, is certainly a $(k\tau,K)^Q$-spanning set. The number of these controls is bounded by $\prod_{j=0}^{k-1} |S_j|$. Hence,%
\begin{equation*}
  \frac{1}{k\tau}\log r_{\inv}(k\tau,K)^Q \leq \log_2(\rme)^{-1}\frac{1}{k\tau}\sum_{j=0}^{k-1}\log_2|S_j|,%
\end{equation*}
implying $\log_2(\rme)\cdot h_{\inv}(K)^Q \leq R_0$ by letting $k\rightarrow\infty$ and $R \rightarrow R_0$.%
\end{proof}

\section{Parameter dependence of spectra}\label{sec_spectra}

Consider the parameter-dependent control-affine system%
\begin{equation}\label{eq_sigmaalpha}
  \Sigma^{\alpha}:\ \dot{x}(t) = f_0(\alpha,x(t)) + \sum_{i=1}^m u_i(t)f_i(\alpha,x(t)),\quad u\in\UC,\ \alpha \in A,%
\end{equation}
where $f_0,f_1,\ldots,f_m:A \tm M \rightarrow TM$ are parameter-dependent vector fields on a Riemannian manifold $M$ with parameter space $A \subset \R^k$, which are $C^{\infty}$ in $(\alpha,x)$. As in Subsect.~\ref{subsec_cas}, we assume that $\UC = L^{\infty}(\R,U)$ for a compact and convex set $U\subset\R^m$ and that all solutions are defined on $\R$. Objects related to $\Sigma^{\alpha}$ are equipped with the superscript $\alpha$, e.g., $\varphi^{\alpha}$ is the transition map of $\Sigma^{\alpha}$.%

In the rest of the paper, we will assume for simplicity that both the parameter space $A$ and the state space $M$ are compact. This makes several arguments simpler, but causes no loss of generality, since we will only consider the action of the control systems $\Sigma^{\alpha}$ on a specified compact subset of $M$ and small neighborhoods of some specified parameter $\alpha^0$.%

In this setting, we have the following result, which is a corollary of \cite[Thm.~2]{CDu} and \cite[Thm.~3.2.28]{CKl}.%

\begin{theorem}\label{thm1}
Let $D^{\alpha^0}$ be a control set of $\Sigma^{\alpha^0}$ with nonempty interior for some $\alpha^0 \in \inner A$. Assume that the following conditions are satisfied:%
\begin{enumerate}
\item[(a)] $\Sigma^{\alpha^0}$ satisfies the Lie algebra rank condition on $\cl D^{\alpha^0}$.%
\item[(b)] The systems $\Sigma^{\alpha}$ are locally accessible for all $\alpha$ in a neighborhood of $\alpha^0$.%
\item[(c)] $\cl D^{\alpha^0}$ coincides with the chain control set $E^{\alpha^0}$ containing $D^{\alpha^0}$.%
\end{enumerate}
Then there exists a neighborhood $W \subset A$ of $\alpha^0$ such that the following hold:%
\begin{enumerate}
\item[(i)] For every $\alpha \in W$ there exists a control set $D^{\alpha}$ of $\Sigma^{\alpha}$ with $\inner D^{\alpha} \cap \inner D^{\alpha^0} \neq \emptyset$.%
\item[(ii)] The map $\alpha \mapsto \cl D^{\alpha}$ from $W$ to the space of nonempty subsets of $M$ is continuous in the Hausdorff metric at $\alpha = \alpha^0$.%
\end{enumerate}
\end{theorem}

The above result can in particular be applied in the case when $\Sigma^{\alpha^0}$ has a uniformly hyperbolic chain control set $E^{\alpha^0}$ with nonempty interior and is locally accessible, because in this case $E^{\alpha^0} = \cl D^{\alpha^0}$ for a control set $D^{\alpha^0}$, as stated in the next theorem whose proof can be found in \cite{CDu}.%

\begin{theorem}\label{thm_closure}
Assume that $E$ is a uniformly hyperbolic chain control set with $\inner E \neq \emptyset$ such that local accessibility holds on $E$. Then $E$ is the closure of a control set $D$.%
\end{theorem}

The next corollary immediately follows.%

\begin{corollary}\label{cor1}
Let $E^{\alpha^0}$ be a regularly uniformly hyperbolic chain control set of $\Sigma^{\alpha^0}$ for some $\alpha^0 \in \inner A$, and denote by $D^{\alpha^0}$ the control set with $\cl D^{\alpha^0} = E^{\alpha^0}$. Assume that the systems $\Sigma^{\alpha}$ are locally accessible for all $\alpha$ in a neighborhood of $\alpha^0$. Then there exists a neighborhood $W \subset A$ of $\alpha^0$ such that the following hold:%
\begin{enumerate}
\item[(i)] For every $\alpha \in W$ there exists a control set $D^{\alpha}$ of $\Sigma^{\alpha}$ with $\inner D^{\alpha} \cap \inner D^{\alpha^0} \neq \emptyset$.%
\item[(ii)] The map $\alpha \mapsto \cl D^{\alpha}$ from $W$ to the space of nonempty subsets of $M$ is continuous in the Hausdorff metric at $\alpha = \alpha^0$.%
\end{enumerate}
\end{corollary}

\subsection{The Morse spectrum over a chain control set}

In this subsection, we prove that the Morse spectrum of an additive cocycle over the lift of a chain control set, considered as a set-valued map on the parameter space $A$, is upper semicontinuous at $\alpha^0 \in A$, provided that also the corresponding chain control sets depend upper semicontinuously on $\alpha$ at $\alpha^0$.% 

\begin{lemma}\label{lem_usclifts}
Let $E^{\alpha}$ be a family of compact chain control sets for $\Sigma^{\alpha}$ such that the set-valued map $\alpha \mapsto E^{\alpha}$ is upper semicontinuous at $\alpha^0$. Then also the set-valued map $\alpha \mapsto \EC^{\alpha}$ is upper semicontinuous at $\alpha_0$.%
\end{lemma}

\begin{proof}
Assume to the contrary that there exists a neighborhood $V$ of $\EC^{\alpha^0}$ such that for every $\delta>0$ there is $\alpha\in A$ with $\|\alpha - \alpha^0\|<\delta$ and $\EC^{\alpha} \not\subset V$. Then there are a sequence $\alpha^k \rightarrow \alpha$ and $(u^k,x^k) \in \EC^{\alpha^k}$ such that $(u^k,x^k) \notin V$. However, by assumption we have $\dist(x^k,E^{\alpha^0}) \rightarrow 0$. By compactness, we may assume $x^k \rightarrow x \in E^{\alpha^0}$ and $u^k \rightarrow u\in\UC$. Since $\varphi^{\alpha^k}(t,x^k,u^k) \in E^{\alpha^k}$ for all $t\in\R$, we have $\varphi^{\alpha^0}(t,x,u) = \lim_{k\rightarrow\infty}\varphi^{\alpha^k}(t,x^k,u^k) \in E^{\alpha^0}$ for all $t\in\R$, again by the assumption. This implies $(u,x) \in \EC^{\alpha^0}$, in contradiction to $(u^k,x^k) \notin V$ for all $k$.%
\end{proof}

The proof of the following proposition is based on ideas from the proof of \cite[Thm.~5.3.10]{CKl}.%

\begin{proposition}\label{prop_usc}
Consider the parametrized control-affine system \eqref{eq_sigmaalpha}. For each $\alpha\in A$ let $E^{\alpha}$ be a compact chain control set of $\Sigma^{\alpha}$, such that $\alpha \mapsto E^{\alpha}$ is upper semicontinuous at some $\alpha^0\in A$. Moreover, let $\gamma^{\alpha}$ be an additive cocycle for $\Phi^{\alpha}$ over $\EC^{\alpha}$ such that $(t,\alpha,u,x) \mapsto \gamma^{\alpha}_t(u,x)$ is continuous on $\R \tm \{\alpha^0\} \tm \EC^{\alpha^0}$. Then the Morse spectrum over $\EC^{\alpha}$ is upper semicontinuous at $\alpha^0$:%
\begin{equation}\label{eq_morsespec_incl}
  \limsup_{\alpha \rightarrow \alpha^0}\Sigma_{\Mo}(\gamma^{\alpha},\EC^{\alpha}) \subset \Sigma_{\Mo}(\gamma^{\alpha^0},\EC^{\alpha^0}).%
\end{equation}
As a consequence, the real-valued map $\alpha \mapsto \inf\Sigma_{\Mo}(\gamma^{\alpha},\EC^{\alpha})$ is lower semicontinuous at $\alpha^0$.%
\end{proposition}

\begin{proof}
By the continuity assumption, we find that $\gamma^{\alpha}_t(u,x)$ is bounded when $\alpha$ is close to $\alpha^0$ and $t$ is bounded. Using the additivity of the cocycles $\gamma^{\alpha}$, this easily implies that $\Sigma_{\Mo}(\gamma^{\alpha},\EC^{\alpha})$ is contained in some compact interval for $\alpha$ close to $\alpha^0$. Hence, \eqref{eq_morsespec_incl} is equivalent to the upper semicontinuity at $\alpha^0$.%

Now let $\lambda$ be an element of the left-hand side of \eqref{eq_morsespec_incl}. This means, there exist sequences $\alpha^k \rightarrow \alpha^0$ in $A$ and $\lambda^k \in \Sigma_{\Mo}(\gamma^{\alpha^k},\EC^{\alpha^k})$ such that $\lambda^k \rightarrow \lambda$. Since $\Sigma_{\Mo}(\gamma^{\alpha^0},\EC^{\alpha^0})$ is closed, it suffices to show that for every $\delta>0$ and $\ep,T>0$ there is a periodic $(\ep,T)$-chain $\zeta$ in $\EC^{\alpha^0}$ for the control flow $\Phi^{\alpha^0}$ with%
\begin{equation}\label{eq_5319}
  |\lambda - \gamma^{\alpha^0}(\zeta)| < \delta.%
\end{equation}
For brevity, in the following we write $\EC^k = \EC^{\alpha^k}$, $\Phi^k = \Phi^{\alpha^k}$ and $\gamma^k = \gamma^{\alpha^k}$. Choose $\delta>0$, $\ep>0$ and $T>1$. Then there is $k_0\in\N$ such that for all $k\geq k_0$ the following holds: For every $(u,x)\in\EC^k$ there is $(v,y)\in\EC^0$ such that%
\begin{equation}\label{eq_5320}
  d(\Phi^k_t(u,x),\Phi^0_t(v,y)) < \frac{\ep}{3} \mbox{\quad for all\ } t\in[0,2T].%
\end{equation}
This follows from uniform continuity of $(t,\alpha,u,x) \mapsto \Phi^{\alpha}_t(u,x)$ on the compact set $[0,2T] \tm A \tm K$, where $K$ is an appropriately chosen compact neighborhood of $\EC^0$, and upper semicontinuity of $\alpha \mapsto \EC^{\alpha}$ at $\alpha^0$ (see Lemma \ref{lem_usclifts}). Similarly, we may choose $k_0$ large enough such that for all $k\geq k_0$,%
\begin{equation}\label{eq_5321}
  \left|\gamma^k_t(u,x) - \gamma^0_t(v,y)\right| < \frac{\delta}{2} \mbox{\quad for all\ } t \in [0,2T].%
\end{equation}
Indeed, assume to the contrary that there are sequences $(u^k,x^k) \in \EC^k$ and $(v^k,y^k) \in \EC^0$ with $d((u^k,x^k),(v^k,y^k)) \rightarrow 0$ and $t_k \in [0,2T]$ with $|\gamma^k_{t_k}(u^k,x^k) - \gamma^0_{t_k}(v^k,y^k)| \geq \delta/2$ for arbitrarily large $k$. By compactness of $[0,2T]$ and $\EC^0$ we may assume $t_k \rightarrow t\in[0,2T]$ and $(u^k,x^k),(v^k,y^k) \rightarrow (u^*,x^*) \in \EC^0$. By the continuity assumption on $\gamma$, this implies $\gamma^k_{t_k}(u^k,x^k) \rightarrow \gamma^0_t(u^*,x^*)$ and $\gamma^0_{t_k}(v^k,y^k) \rightarrow \gamma^0_t(u^*,x^*)$, a contradiction.%

From $\lambda^k \rightarrow \lambda$ with $\lambda^k \in \Sigma_{\Mo}(\gamma^k,\EC^k)$ it follows that for $k$ large enough there are periodic $(\ep/3,T)$-chains $\zeta^k$ in $\EC^k$ for $\Phi^k$, given by $n^k\in\N$, $T_0^k,\ldots,T_{n^k-1}^k \geq T$, $(v_0^k,y_0^k),\ldots,(v_{n^k}^k,y_{n^k}^k) \in \EC^k$ with%
\begin{equation}\label{eq_5322}
  d(\Phi^k_{T_j^k}(v^k_j,y^k_j),(v^k_{j+1},y^k_{j+1})) < \frac{\ep}{3} \mbox{\quad for\ } j = 0,1,\ldots,n^k-1%
\end{equation}
and%
\begin{equation}\label{eq_5323}
  |\lambda - \gamma^k(\zeta^k)| < \frac{\delta}{2}.%
\end{equation}
Here we use that the sets $\EC^k$ are chain transitive, and hence the Morse spectra over these sets can be obtained from periodic chains only. Now fix $k$ and partition each interval $[0,T_j^k]$, $j=0,\ldots,n^k-1$, into pieces of length $\tau_{ji}$, $i=0,\ldots,l_j-1$, with $T \leq \tau_{ji} \leq 2T$. Then%
\begin{equation*}
  T_j^k = \sum_{i=0}^{l_j-1}\tau_{ji} \mbox{\quad and\quad} T_j^k \geq l_jT > l_j.%
\end{equation*}
Define $\xi_0^j := (v_j^k,y_j^k)$ and $\xi_{i+1}^j := \Phi^k_{\tau_{ji}}(\xi_i^j)$ for $i=0,\ldots,l_j-2$. Via \eqref{eq_5320} and \eqref{eq_5321} replace each point $\xi_i^j$ with a corresponding point $\tilde{\xi}_i^j \in \EC^0$. We claim that the points $\tilde{\xi}_0^0,\ldots,\tilde{\xi}_{l_0-1}^0,\tilde{\xi}_0^1,\ldots,\tilde{\xi}_{l_1-1}^1,\ldots,\tilde{\xi}_0^{n_k-1},\ldots,\tilde{\xi}_{l_k-1}^{n_k-1}, \tilde{\xi}_{0}^{n_k}$ together with the times $\tau_{ji} \geq T$ form an $(\ep,T)$-chain $\zeta^0$ in $\EC^0$ for $\Phi^0$. Indeed, if $i \in \{0,1,\ldots,l_j-2\}$, then%
\begin{equation*}
  d(\Phi^0_{\tau_{ji}}(\tilde{\xi}_i^j),\tilde{\xi}_{i+1}^j) \leq d(\Phi^0_{\tau_{ji}}(\tilde{\xi}_i^j),\Phi^k_{\tau_{ji}}(\xi_i^j)) + d(\Phi^k_{\tau_{ji}}(\xi_i^j),\xi_{i+1}^j) + d(\xi_{i+1}^j,\tilde{\xi}_{i+1}^j)
	< \frac{\ep}{3} + 0 + \frac{\ep}{3} < \ep%
\end{equation*}
and (using \eqref{eq_5322})%
\begin{eqnarray*}
  d(\Phi^0_{\tau_{jl_j-1}}(\tilde{\xi}_{l_j-1}^j),\tilde{\xi}_0^{j+1}) &\leq& d(\Phi^0_{\tau_{jl_j-1}}(\tilde{\xi}_{l_j-1}^j),\Phi^k_{\tau_{jl_j-1}}(\xi_{l_j-1}^j)) + d(\Phi^k_{\tau_{jl_j-1}}(\xi_{l_j-1}^j),\xi_0^{j+1})\\
	\quad && + d(\xi_0^{j+1},\tilde{\xi}_0^{j+1})\\
	                                                                    &<& \frac{\ep}{3} + d(\Phi^k_{T_j^k}(v_j^k,y_j^k),(v_{j+1}^k,y_{j+1}^k)) + \frac{\ep}{3} < \ep.%
\end{eqnarray*}
Now we compare the Morse exponents $\gamma^0(\zeta^0)$ and $\gamma^k(\zeta^k)$, using \eqref{eq_5321}:%
\begin{eqnarray*}
  \gamma^0(\zeta^0) &=&  \frac{1}{\sum_j T_j^k}\sum_{j=0}^{n^k-1}\sum_{i=0}^{l_j-1}\gamma^0(\tau_{ji},\tilde{\xi}_i^j) < \frac{1}{\sum_j T_j^k}\sum_{j=0}^{n^k-1}\sum_{i=0}^{l_j-1}\left(\gamma^k(\tau_{ji},\xi_i^j) + \frac{\delta}{2}\right)\\
	&=& \frac{1}{\sum_j T_j^k}\left(\frac{\delta}{2}\sum_{j=0}^{n^k-1}l_j + \sum_{j=0}^{n^k-1} \gamma^k(T_j^k,(v_j^k,y_j^k))\right) < \gamma^k(\zeta^k) + \frac{\delta}{2},%
\end{eqnarray*}
where we use that $l_j < T_j^k$. The estimate in the other direction works analogously. Hence,%
\begin{equation*}
  |\lambda - \gamma^0(\zeta^0)| \leq |\lambda - \gamma^k(\zeta^k)| + |\gamma^k(\zeta^k) - \gamma^0(\zeta^0)| < \frac{\delta}{2} + \frac{\delta}{2} = \delta,%
\end{equation*}
which completes the proof of the upper semicontinuity statement. Lower semicontinuity of $\alpha \mapsto \inf\Sigma_{\Mo}(\gamma^{\alpha},\EC^{\alpha})$ immediately follows.%
\end{proof}

\subsection{The Floquet spectrum over a control set}

Let $D$ be a relatively compact control set with nonempty interior of the control-affine system \eqref{eq_cas} and write $\DC$ for its lift to $\UC \tm M$. Then we define the \emph{Floquet spectrum} $\Sigma_{\Fl}(\gamma,\DC)$ for an additive cocycle $\gamma$ over $\DC$ by%
\begin{equation*}
  \Sigma_{\Fl}(\gamma,\DC) := \left\{ \lambda\in\R\ :\ \exists \mbox{ periodic } (u,x) \in \UC \tm \inner D \mbox{ with } \lambda = \lambda((u,x);\gamma)\right\}.%
\end{equation*}
Obviously, the Floquet spectrum is contained in the Lyapunov spectrum. Moreover, under the assumption of local accessibility on $\inner D$, the Floquet spectrum is nonempty, because periodic orbits can be constructed through every point $x\in\inner D$.%

The following proposition will allow us to study the dependence of the Floquet spectrum on parameters for the additive cocycle given by \eqref{eq_unstabledet_cocycle}.%

\begin{proposition}\label{prop_floquet}
Let $E$ be a regularly uniformly hyperbolic chain control set of the control-affine system \eqref{eq_cas} and consider the additive cocycle \eqref{eq_unstabledet_cocycle} over $\EC$. Furthermore, consider the subadditive cocycle%
\begin{equation*}
  \kappa:\R \tm \EC \rightarrow \R,\quad \kappa_t(u,x) = \log^+\left\|\rmd\varphi_{t,u}(x)^{\wedge}\right\|.%
\end{equation*}
Let $D$ be the control set with $\cl D = E$. Then%
\begin{equation*}
  \Sigma_{\Fl}(\gamma,\DC) = \left\{ \lim_{t\rightarrow\infty}\frac{1}{t}\kappa_t(u,x) : (u,x) \in \UC \tm M \mbox{ periodic},\ x \in \inner D \right\}.%
\end{equation*}
\end{proposition}

\begin{proof}
It suffices to show that%
\begin{equation*}
  \lim_{t\rightarrow\infty}\frac{1}{t}\kappa_t(u,x) = \frac{1}{\tau}\gamma_{\tau}(u,x),%
\end{equation*}
when $\tau>0$ is the period of $(u,x)$. This follows from \cite[Prop.~3]{DSK}.%
\end{proof}

The subadditive cocycle $\kappa$ defined above has the nice property that $\kappa_t(u,x)$ only depends on $u|_{[0,t]}$ for $t>0$ and not on the entire function $u$. This property is very helpful in many arguments. Altogether, this justifies the following definitions.%

\begin{definition}
The {\bf Floquet spectrum} of a subadditive cocycle $\kappa$ over the lift $\DC$ of a control set $D$ is defined as%
\begin{equation*}
  \Sigma_{\Fl}(\kappa,\DC) := \left\{ \lim_{t\rightarrow\infty}\frac{1}{t}\kappa_t(u,x) : (u,x) \in \UC \tm M \mbox{ periodic, } x \in \inner D \right\}.%
\end{equation*}
We say that $\kappa$ has the {\bf restriction property} if, for all $t>0$ and $(u,x),(v,x) \in \DC$, the values $\kappa_t(u,x)$ and $\kappa_t(v,x)$ coincide whenever $u(s) = v(s)$ for almost all $s\in[0,t]$.%
\end{definition}

The following proposition yields the upper semicontinuity of $\inf\Sigma_{\Fl}$ for a parametrized control-affine system, under appropriate assumptions. The proof uses some ideas from \cite[Thm.~6.2.24]{CKl}.%

\begin{proposition}\label{prop_lsc}
We make the following assumptions for the parametrized control-affine system \eqref{eq_sigmaalpha} and some $\alpha^0 \in \inner A$:%
\begin{enumerate}
\item[(i)] For each $\alpha$ in a neighborhood of $\alpha^0$ the system $\Sigma^{\alpha}$ is locally accessible.%
\item[(ii)] There exists a control set $D^{\alpha^0}$ of $\Sigma^{\alpha^0}$ with nonempty interior. For $\alpha$ in a neighborhood of $\alpha^0$ let $D^{\alpha}$ denote the unique control set of $\Sigma^{\alpha}$ with $\inner D^{\alpha^0} \cap \inner D^{\alpha} \neq \emptyset$.%
\item[(iii)] The system $\Sigma^{\alpha^0}$ satisfies the Lie algebra rank condition on $\inner D^{\alpha^0}$.%
\item[(iv)] $\kappa^{\alpha}:\R \tm (\UC \tm M) \rightarrow \R$, $\alpha\in A$, is a family of nonnegative subadditive cocycles for $\Phi^{\alpha}$ with the restriction property such that $(t,\alpha,u,x) \mapsto \kappa^{\alpha}_t(u,x)$ is continuous.%
\end{enumerate}
Then the map $\alpha \mapsto \inf\Sigma_{\Fl}(\kappa^{\alpha},\DC^{\alpha})$ is upper semicontinuous at $\alpha = \alpha^0$.%
\end{proposition}

\begin{proof}
The proposition is proved in three steps.%

\emph{Step 1.} From assumption (i) and \cite[Thm.~3.2.28]{CKl} the existence of control sets $D^{\alpha}$ as specified in assumption (ii) follows. Let $A_0$ denote the neighborhood of $\alpha^0$, where these control sets exist. We fix $\ep>0$ and write%
\begin{equation*}
  i(\alpha) := \inf\Sigma_{\Fl}(\kappa^{\alpha},\DC^{\alpha}) \mbox{\quad for all\ } \alpha\in A_0.%
\end{equation*}
Then there is a $\Phi^{\alpha^0}$-periodic point $(u_0,x_0)$ with $x_0 \in \inner D^{\alpha^0}$ of some period $\tau_0>0$ such that%
\begin{equation}\label{eq_lsc_1}
  \frac{1}{t}\kappa^{\alpha^0}_t(u_0,x_0) < i(\alpha^0) + \ep \mbox{\quad for all\ } t \geq t_0,%
\end{equation}
where $t_0 > 0$ is chosen sufficiently large. We choose a compact neighborhood $K$ of $x_0$ with $K \subset \inner D^{\alpha^0}$. It is well-known that the first hitting time, i.e., the minimal time to steer from any $x\in K$ to any $y\in K$ is bounded by some $H<\infty$ (see \cite[Lem.~3.2.21]{CKl}).%

\emph{Step 2.} Let $S>0$ be chosen small enough so that $\OC^{\alpha^0,-}_{\leq S}(x_0) \subset K$. We show that there exists an open set $V \subset M$ so that $V \subset \OC^{\alpha,-}_{\leq S}(x_0)$ for all $\alpha$ in a neighborhood $A_1 \subset A_0$ of $\alpha^0$.%

From the Lie algebra rank condition at $x_0$ it follows that there exist $u_1,\ldots,u_d$ ($d=\dim M$) such that the following map is continuously differentiable and its partial derivative w.r.t.~the second component has full rank $d$ at some point of the form $(\alpha^0,\tau^0)$ with $\tau^0 = (\tau_1,\ldots,\tau_d)$ and $\tau_i>0$ arbitrarily small, say $<S/(2d)$:%
\begin{equation*}
  \psi:A \tm (0,\infty)^d \rightarrow M,\quad (\alpha,(\tau_1,\ldots,\tau_d)) \mapsto \varphi^{\alpha}\Bigl(-\sum_{i=1}^d\tau_i,x_0,u^{\tau}\Bigr),%
\end{equation*}
where $u^{\tau}(t) = u_i$ on $(\tau_0+\tau_1+\ldots+\tau_{i-1},\tau_0+\tau_1+\ldots+\tau_i)$ with $\tau_0 = 0$ for $i=1,\ldots,d$ (for these statements, see \cite[Thm.~9 in Ch.~4]{Son}). By the inverse function theorem, we find open neighborhoods $N \subset \R^d$ of $\tau$ and $W \subset M$ of $\psi(\alpha^0,\tau)$ such that $\psi(\alpha^0,\cdot)$ maps $N$ to $W$ diffeomorphically. For every $y\in W$, the implicit function theorem applied to $\psi$ guarantees that $y \in \OC^{\alpha,-}_{\leq S}(x_0)$ for all $\alpha$ sufficiently close to $\alpha^0$, say $\|\alpha-\alpha^0\|<\delta$. By possibly shrinking $W$ to a smaller open set $V$, we can choose $\delta$ independently of $y$. Hence, we have found an open set $V \subset \OC^{\alpha,-}_{\leq S}(x_0)$ for all $\|\alpha-\alpha^0\|<\delta$. Let $A_1 := A_0 \cap B(\alpha^0,\delta)$.%

\emph{Step 3.} We choose $n\in\N$ large enough so that%
\begin{equation}\label{eq_lsc_2}
  n\tau_0 \geq t_0 \mbox{\quad and\quad} \frac{1}{n\tau_0}\kappa^{\alpha}_t(u,x) \leq \ep \mbox{\quad for all\ } t\in [0,H+S] \mbox{ and } (\alpha,u,x) \in A \tm \UC \tm M,%
\end{equation}
where we use our assumption that $A$ and $M$ are compact so that $\kappa_t^{\alpha}(u,x)$ is bounded on $[0,H+S] \tm A \tm \UC \tm M$. We choose a neighborhood $A_2 \subset A_1$ of $\alpha^0$ so that%
\begin{equation}\label{eq_lsc_3}
  \frac{1}{n\tau_0}|\kappa^{\alpha}_{n\tau_0}(u_0,x_0) - \kappa^{\alpha^0}_{n\tau_0}(u_0,x_0)| < \ep \mbox{\quad for all\ } \alpha \in A_2.%
\end{equation}
Now we pick $\alpha\in A_2$ and construct a periodic point for $\Phi^{\alpha}$ as follows. First steer from $x_0$ to $x_0$ with $\Sigma^{\alpha^0}$ in time $n\tau_0$ by applying $u_0$ $n$ times. Let $u_1$ be a control so that $\varphi^{\alpha^0}(t_1,x_0,u_1) \in V$ with $t_1 \leq H$, which is possible, because $V \subset K$. By continuous dependence of the solutions on $\alpha$, we find that $\varphi^{\alpha}(n\tau_0+t_1,x_0,u_0^nu_1) \in V$ for all $\alpha$ in a neighborhood $A_3 \subset A_2$ of $\alpha^0$.%

By Step 2, there exists another control $u_2$ and a time $t_2 \leq S$ with $\varphi^{\alpha}(t_2,\varphi^{\alpha}(n\tau_0+t_1,x_0,u_0^nu_1),u_2) = x_0$. We put $u := u_0^nu_1u_2$, which is a control defined on $[0,n\tau_0+t_1+t_2]$, and we extend this control $(n\tau_0+t_1+t_2)$-periodically to $\R$. This yields the periodic point $(u,x_0)$ for $\Sigma^{\alpha}$ with period $T:=n\tau_0+t_1+t_2$. Using subadditivity, nonnegativity and the restriction property for $\kappa^{\alpha}$, we obtain%
\begin{eqnarray*}
  \frac{1}{T}\kappa^{\alpha}_T(u,x_0) &\leq& \frac{1}{T}\kappa^{\alpha}_{n\tau_0}(u_0,x_0) + \frac{1}{T}\kappa^{\alpha}_{t_1+t_2}(u_1u_2,\varphi^{\alpha}(n\tau_0,x_0,u_0))\\
	&\stackrel{\eqref{eq_lsc_2}}{\leq}& \frac{1}{n\tau_0}\kappa^{\alpha}_{n\tau_0}(u_0,x_0) + \ep \stackrel{\eqref{eq_lsc_3}}{\leq} \frac{1}{n\tau_0}\kappa^{\alpha^0}_{n\tau_0}(u_0,x_0) + 2\ep \stackrel{\eqref{eq_lsc_1}}{<} i(\alpha^0) + 3\ep.%
\end{eqnarray*}
Now we find that%
\begin{eqnarray*}
  i(\alpha) &\leq& \lim_{t\rightarrow\infty}\frac{1}{t}\kappa^{\alpha}_t(u,x_0) = \lim_{\N\ni m\rightarrow\infty}\frac{1}{mT} \kappa^{\alpha}_{mT}(u,x_0)\\
	&=& \inf_{m\in\N}\frac{1}{mT}\kappa^{\alpha}_{mT}(u,x_0) \leq \frac{1}{T}\kappa^{\alpha}_T(u,x_0) < i(\alpha^0) + 3\ep,%
\end{eqnarray*}
where we use that the sequence $a_m := \kappa^{\alpha}_{mT}(u,x_0)$, $m\in\N$, is subadditive. This implies $\limsup_{\alpha\rightarrow\alpha^0}i(\alpha) \leq i(\alpha^0)$, which completes the proof.%
\end{proof}

\section{The main result}\label{sec_mainres}

In this section, we make use of the additive cocycle%
\begin{equation*}
  \gamma^{\alpha}_t(u,x) = \log J^+\rmd\varphi^{\alpha}_{t,u}(x)%
\end{equation*}
over a uniformly hyperbolic chain control set $E^{\alpha}$ of $\Sigma^{\alpha}$ and the subadditive cocycle%
\begin{equation*}
  \kappa^{\alpha}_t(u,x) = \log^+\|\rmd\varphi^{\alpha}_{t,u}(x)^{\wedge}\|,%
\end{equation*}
defined over $\UC \tm M$ for each of the systems $\Sigma^{\alpha}$, $\alpha \in A$. Observe that the map $(t,\alpha,u,x) \mapsto \kappa^{\alpha}_t(u,x)$ is continuous, because $\rmd\varphi^{\alpha}_{t,u}(x)$ varies continuously with $(t,\alpha,u,x)$. (This is shown with arguments that can be found in \cite[Prop.~1.17 and Thm.~1.1]{Ka1}).%

\begin{proposition}\label{prop_robustness}
In Corollary \ref{cor1}, let $E^{\alpha}$ denote the unique chain control set with $D^{\alpha} \subset E^{\alpha}$. Then, for $\alpha$ sufficiently close to $\alpha^0$, $E^{\alpha}$ is uniformly hyperbolic, its lift $\EC^{\alpha}$ is an isolated invariant set and $\EC^{\alpha}$ is a graph over $\UC$. In particular, $E^{\alpha} = \cl D^{\alpha}$. If we additionally assume that the Lie algebra rank condition holds on $\inner E^{\alpha}$, then%
\begin{equation*}
  h^*_{\inv}(E^{\alpha}) = \inf\Sigma_{\Mo}(\gamma^{\alpha},\EC^{\alpha}) = \inf\Sigma_{\Fl}(\kappa^{\alpha},\DC^{\alpha}).%
\end{equation*}
\end{proposition}

\begin{proof}
Consider the time-$1$-discretizations of the control flows $\Phi^{\alpha}$, i.e., the maps $\Phi^{\alpha}_1:\UC \tm M \rightarrow \UC \tm M$. These maps satisfy the assumptions imposed on the map $f$ in Subsection \ref{subsec_shadowing}. Moreover, the set $E^{\alpha^0}$ is an isolated invariant uniformly hyperbolic set for $\Phi^{\alpha^0}_1$ as in Lemma \ref{lem_shadowing}. Part (iii) of this lemma implies the existence of an isolated invariant uniformly hyperbolic set $\FC^{\alpha} \subset \UC \tm M$ for $\Phi^{\alpha}_1$, when $\alpha$ is close enough to $\alpha^0$. Consider some $(u,x) \in \FC^{\alpha}$ and assume to the contrary that $\Phi^{\alpha}_t(u,x) \notin \FC^{\alpha}$ for some $t\in\R$. Replacing $(u,x)$ with another point in the orbit $\{\Phi^{\alpha}_t(u,x) : t \in \R\}$, we can assume that $t>0$ is small enough so that $\Phi^{\alpha}_k(u,x) \in \FC^{\alpha}$ for all $k\in\Z$ and $\Phi^{\alpha}_{k+t}(u,x)$ is contained in an isolating neighborhood of $\FC^{\alpha}$ for all $k\in\Z$, leading to the contradiction $\Phi^{\alpha}_t(u,x) \in \FC^{\alpha}$. Hence, $\FC^{\alpha}$ is also an isolated invariant set of the control flow $\Phi^{\alpha}$.%

We can show that $\FC^{\alpha} = \EC^{\alpha}$ as follows. For $\alpha$ close enough to $\alpha^0$, both $\FC^{\alpha}$ and $\EC^{\alpha}$ are contained in a neighborhood $W$ of $\EC^{\alpha^0}$, where the shadowing lemma applies, observing that $\alpha \mapsto \EC^{\alpha}$ is upper semicontinuous at $\alpha = \alpha^0$ (see \cite[Cor.~3.4.7]{CKl} and Lemma \ref{lem_usclifts}). As a chain transitive set for the flow $\Phi^{\alpha}$ containing periodic orbits, the set $\EC^{\alpha}$ is also chain transitive for the map $\Phi^{\alpha}_1$ (see \cite[Prop.~3.1.11]{CK2}). In particular, every point $(u,x) \in \EC^{\alpha}$ is the initial point of a bi-infinite $\delta$-pseudo-orbit in $W$ for $\Phi^{\alpha}_1$, where $\delta>0$ can be chosen arbitrarily. By Lemma \ref{lem_shadowing} we know that for every $\ep>0$ there is a $\delta>0$ so that every bi-infinite $\delta$-pseudo-orbit is $\ep$-shadowed by a point in $\FC^{\alpha}$. This implies that $\EC^{\alpha} \subset \cl\FC^{\alpha} = \FC^{\alpha}$.%

With the same argumentation as in \cite[Thm.~3.1]{MZh} we can show that the restriction of $\Phi^{\alpha^0}_1$ to $\EC^{\alpha^0}$ is topologically conjugate to the restriction of $\Phi^{\alpha}_1$ to $\FC^{\alpha}$, using the assumption that $\EC^{\alpha^0}$ is a graph over $\UC$. The topological conjugacy is of the form $H(u,x) = (u,h(u,x))$, $H:\EC^{\alpha^0} \rightarrow \FC^{\alpha}$. (A conjugacy of this form is called a topological skew-conjugacy in \cite{MZh}.) This immediately implies that also $\FC^{\alpha}$ is a graph over $\UC$. Moreover, it implies that $\FC^{\alpha}$ is chain transitive, and hence, by maximality, $\FC^{\alpha} = \EC^{\alpha}$. It is not hard to show that $E^{\alpha}$ is a uniformly hyperbolic chain control set, using that $\FC^{\alpha}$ is a uniformly hyperbolic set for $\Phi^{\alpha}_1$. Finally, the formula for the invariance entropy of $E^{\alpha}$ follows from the following considerations. From Theorem \ref{thm_iedr_1} we know that%
\begin{equation*}
  h^*_{\inv}(E^{\alpha}) = \inf_{(u,x) \in \EC^{\alpha}}\limsup_{t\rightarrow\infty}\frac{1}{t}\gamma_t^{\alpha}(u,x).%
\end{equation*}
By the proof of \cite[Thm.~5.4]{DSK}, we obtain the same infimum if we only consider periodic $(u,x) \in \inner\UC \tm \inner D$. Since the endpoints of the Morse spectrum are Lyapunov exponents, it follows that%
\begin{equation*}
  h^*_{\inv}(E^{\alpha}) = \inf\Sigma_{\Fl}(\gamma^{\alpha},\DC^{\alpha}) = \inf\Sigma_{\Ly}(\gamma^{\alpha},\EC^{\alpha}) = \inf\Sigma_{\Mo}(\gamma^{\alpha},\EC^{\alpha}).%
\end{equation*}
Finally, from Proposition \ref{prop_floquet} it follows that $\inf\Sigma_{\Fl}(\gamma^{\alpha},\DC^{\alpha}) = \inf\Sigma_{\Fl}(\kappa^{\alpha},\DC^{\alpha})$.%
\end{proof}

Now we are ready to prove the main result of the paper.%

\begin{theorem}
Consider the parametrized control-affine system \eqref{eq_sigmaalpha} and some $\alpha^0 \in \inner A$. Let $E^{\alpha^0}$ be a regularly uniformly hyperbolic chain control set of $\Sigma^{\alpha^0}$ and assume that $\Sigma^{\alpha}$ satisfies the Lie algebra rank condition for all $\alpha$ in a neighborhood of $\alpha^0$. Then there exists a neighborhood $W$ of $\alpha^0$ such that for every $\alpha\in W$ there exists a chain control set $E^{\alpha}$ of $\Sigma^{\alpha}$, so that the map%
\begin{equation*}
  \alpha \mapsto h_{\inv}^*(E^{\alpha})%
\end{equation*}
is well-defined and continuous at $\alpha = \alpha^0$.%
\end{theorem}

\begin{proof}
By Corollary \ref{cor1}, there exists a neighborhood $W_0$ of $\alpha^0$ so that for every $\alpha \in W_0$ a control set $D^{\alpha}$ of $\Sigma^{\alpha}$ exists with $\inner D^{\alpha} \cap \inner D^{\alpha^0} \neq \emptyset$ and $\alpha \mapsto \cl D^{\alpha}$ is continuous at $\alpha=\alpha^0$ in the Hausdorff metric. From the assumption that the Lie algebra rank condition holds for $\alpha$ close to $\alpha^0$, it follows that there exist unique chain control sets $E^{\alpha}$ with $D^{\alpha} \subset E^{\alpha}$ for all $\alpha$ in a neighborhood $W_1 \subset W_0$ of $\alpha^0$. By Proposition \ref{prop_robustness}, there exists another neighborhood $W_2 \subset W_1$ of $\alpha^0$ so that $E^{\alpha} = \cl D^{\alpha}$ for $\alpha \in W_2$ and%
\begin{equation*}
  h_{\inv}^*(E^{\alpha}) = \inf\Sigma_{\Mo}(\gamma^{\alpha},\EC^{\alpha}) = \inf\Sigma_{\Fl}(\kappa^{\alpha},\DC^{\alpha}).%
\end{equation*}
Since all assumptions of Proposition \ref{prop_lsc} are satisfied, the map $\alpha \mapsto h_{\inv}^*(E^{\alpha})$ is upper semicontinuous at $\alpha = \alpha^0$. Now we prove that the continuity assumption of Proposition \ref{prop_usc} is satisfied for $(t,\alpha,u,x) \mapsto \gamma^{\alpha}_t(u,x)$, implying the lower semicontinuity, and hence the assertion of the theorem. To this end, it is sufficient to show that the unstable subspace $E^{\alpha,+}_{u,x}$ varies continuously with $(\alpha,u,x)$ on $\{\alpha^0\} \tm \EC^{\alpha^0}$. This follows from the fact that for a sufficiently small compact neighborhood $A_0$ of $\alpha^0$, the set%
\begin{equation*}
  \EC := \{ (\alpha,u,x) \in A \tm \UC \tm M\ :\ (u,x) \in \EC^{\alpha},\ \alpha \in A_0 \}%
\end{equation*}
is a compact invariant uniformly hyperbolic set for the extended skew-product flow $\Psi_t(\alpha,u,x) = (\alpha,\Phi^{\alpha}_t(u,x))$, $t\in\R$, with base space $A_0 \tm \UC$. Here we use the following three facts:%
\begin{enumerate}
\item[(i)] By changing the parameter $\alpha$ slightly, the constants used in the definition of uniform hyperbolicity also change slightly (see, e.g., the proof of \cite[Prop.~6.4.6]{KHa}).%
\item[(ii)] The continuous dependence of the unstable subspace is a consequence of the uniform contraction and expansion estimates (see, e.g., \cite[Lem.~6.4]{Ka1}).%
\item[(iii)] The set $\EC$ is closed, because the map $\alpha \mapsto E^{\alpha}$ is not only continuous (in the Hausdorff metric) at $\alpha = \alpha^0$, but on a whole neighborhood of $\alpha^0$. This holds, because the assumptions of Theorem \ref{thm1} are satisfied for every $\alpha$ in a neighborhood of $\alpha^0$.%
\end{enumerate}
The proof is complete.%
\end{proof}

\begin{corollary}
Under the assumptions of the preceding theorem, $E^{\alpha}$ is a regularly uniformly hyperbolic chain control set for all $\alpha$ in a neighborhood $A_0$ of $\alpha^0$ and $\alpha \mapsto h_{\inv}^*(E^{\alpha})$ is continuous on $A_0$.%
\end{corollary}

\end{document}